\documentclass{amsart}

\usepackage{amsmath,amsthm,amssymb}

\newcommand{\fn}[1]{\mathrm{#1}}
\newcommand{\mdl}[1]{\mathcal{#1}}

\newcommand{\ph}{\varphi}
\newcommand{\NN}{\mathbb{N}}
\newcommand{\QQ}{\mathbb{Q}}

\newcommand{\st}{\; | \;}
\newcommand{\len}{\fn{length}}

\title{Uncomputably noisy ergodic limits}

\author{Jeremy Avigad}

\address{Departments of Philosophy and Mathematical Sciences\\
Carnegie Mellon University\\
Pittsburgh, PA 15213}


\thanks{Work partially supported by NSF grant DMS-1068829. I am grateful to an anonymous referee for comments and suggestions.}

\begin{document}

\begin{abstract}
V'yugin \cite{vyugin:97,vyugin:98} has shown that there are a computable shift-invariant measure on $2^{\mathbb N}$ and a simple function $f$ such that there is no computable bound on the rate of convergence of the ergodic averages $A_n f$. Here it is shown that in fact one can construct an example with the property that there is no computable bound on the complexity of the limit; that is, there is no computable bound on how complex a simple function needs to be to approximate the limit to within a given $\varepsilon$.
\end{abstract}

\maketitle

Let $2^\NN$ denote Cantor space,  the space of functions from $\NN$ to the discrete space $\{0, 1\}$ under the product topology. Viewing elements of this space as infinite sequences, for any finite sequence $\sigma$ of $0$'s and $1$'s let $[\sigma]$ denote the set of elements of $2^\NN$ that extend $\sigma$. The collection $\mdl B$ of Borel sets in the standard topology are generated by the set of such $[\sigma]$. For each $k$, let $\mdl B_k$ denote the finite $\sigma$-algebra generated by the partition $\{ [\sigma] \st \len(\sigma) = k \}$. If a function $f$ from $2^{\mathbb N}$ to $\mathbb Q$ is measurable with respect to $\mdl B_k$, I will call it a \emph{simple function} with \emph{complexity at most $k$}. 

Let $\mu$ be any probability measure on $(2^\NN, \mdl B)$, and let $f$ be any element of $L^1(\mu)$. Say that a function $k$ from $\mathbb Q^+$ to $\mathbb N$ is a \emph{bound on the complexity of $f$} if, for every $\varepsilon > 0$, there is a simple function $g$ of complexity at most $k(\varepsilon)$ such that $\| f - g \| < \varepsilon$. If $(f_n)$ is any convergent sequence of elements of $L^1(\mu)$ with limit $f$, say that $r(\varepsilon)$ is a \emph{bound on the rate of convergence of $(f_n)$} if for every $n \geq r(\varepsilon)$, $\| f_n - f \| < \varepsilon$. (One can also consider rates of convergence in any of the $L^p$ norms for $1 < p < \infty$, or in measure. Since all the sequences considered below are uniformly bounded, this does not affect the results below.)

Now suppose that $\mu$ is a computable measure on $2^\NN$ in the sense of computable measure theory \cite{hoyrup:rojas:09b,weihrauch:99}. Then if $f$ is any computable element of $L^1(\mu)$, there is a computable sequence $(f_n)$ of simple functions that approaches $f$ with a computable rate of convergence $r(\varepsilon)$; this is essentially what it \emph{means} to be a computable element of $L^1(\mu)$. In particular, setting $k(\varepsilon)$ equal to the complexity of $f_{r(\varepsilon)}$ provides a computable bound on the complexity of $f$. But the converse need not hold: if $r$ is any noncomputable real number and $f$ is the constant function with value $r$, then $f$ is not computable even though there is a trivial bound on its complexity. 

It is not hard to compute a sequence of simple functions $(f_n)$ that converges to a function $f$ even in the $L^\infty$ norm with the property that there is no computable bound on the complexity of the limit, with respect to the standard coin-flipping measure on $2^\NN$. Notice that this is stronger than saying that there is no computable bound on the rate of convergence of $(f_n)$ to $f$; it says that there is no way of computing bounds on the complexity of \emph{any} sequence of good approximations to $f$. 

To describe such a sequence, for each $k$, let $h_k$ be the $\mdl B_k$-measurable Rademacher function defined by
\[
h_k = \sum_{\{ \sigma \st \len(\sigma) = k \}} (-1)^{\sigma_{k-1}} 1_{[\sigma]},
\]
where $\sigma_{k-1}$ denotes the last bit of $\sigma$ and $1_{[\sigma]}$ denotes the characteristic function of the cylinder set $[\sigma]$. Intuitively, $h_k$ is a ``noisy'' function of complexity $k$. Finally, let $f_n = \sum_{i \leq n} 4^{-\ph(i)} h_i$, where $\ph$ is an injective enumeration of any computably enumerable set, like the halting problem, that is not computable. Given any $m$, if $n$ is large enough so that $\ph(j) > m$ whenever $j > n$, then for every $i > n$ and every $x$ we have $| f_i(x) - f_n(x) | \leq \sum_{j \geq m} 4^{-j} < 1 / (3 \cdot 4^m)$. Thus the sequence $(f_n)$ converges in the $L^\infty$ norm. At the same time, it is not hard to verify that if $f$ is the $L^1$ limit of this sequence and $g$ is a simple function of complexity at at most $n$ such that $\mu(\{ x \st |g(x) - f(x)|> 4^{-(m+1)} \}) < 1/2$, then $m$ is in the range of $\ph$ if and only if $\ph(j) = n$ for some $j < n$. Thus one can compute the range of $\ph$ from any bound on the complexity of $f$.

The sequence $(f_n)$ just constructed is contrived, and one can ask whether similar sequences arise ``in nature.'' Letting $A_n f$ denote the ergodic average $\frac{1}{n} \sum_{i < n} f \circ T_n$, the mean ergodic theorem implies that for every measure $\mu$ on $2^\NN$ and $f$ in $L^1(\mu)$, the sequence $(A_n f)$ converges in the $L^1$ norm. However, V'yugin \cite{vyugin:97,vyugin:98} has shown that there is a computable shift-invariant measure $\mu$ on Cantor space such that there is no computable bound on the rate of convergence of $(A_n 1_{[1]})$.
In V'yugin's construction, the limit doesn't have the property described in the last paragraph; in fact, it is very easy to bound the complexity of the limit in question, which places a noncomputable mass on the string of $0$'s and the string of $1$'s, and is otherwise homogeneous. The next theorem shows, however, that there are computable measures $\mu$ such that the limit does have this stronger property. 

\newtheorem*{theorem}{Theorem}

\begin{theorem}
 There is a computable shift-invariant measure $\mu$ on $2^\NN$ such that if $f = \lim_n A_n 1_{[1]}$, the halting problem can be computed from any bound on the complexity of $f$.
\end{theorem}

\begin{proof}
If $\sigma$ is any finite binary sequence, let $\sigma^*$ denote the element $\sigma\sigma\sigma\ldots$ of Cantor space. For each $e$, define a measure $\mu_e$ as follows: if Turing machine $e$ halts in $s$ steps, let $\mu_e$ put mass uniformly on these $8s$ elements:
\begin{itemize}
 \item all $4s$ shifts of $(1^s0^{3s})^*$
 \item all $4s$ shifts of $(1^{3s}0^s)^*$
\end{itemize}
Otherwise, let $\mu_e$ divide mass uniformly between $0^*$ and $1^*$. Each measure $\mu_e$ is shift invariant, by construction. I will show, first, that $\mu_e$ is computable uniformly in $e$, which is to say, there is a single algorithm that, given $e$, $\sigma$, and $\varepsilon > 0$, computes $\mu_e([\sigma])$ to with $\varepsilon$. I will then show that information as to the complexity needed to approximate $f$ in $(2^\omega, \mdl B, \mu_e)$ allows one to determine whether or not Turing machine $e$ halts. The desired conclusion is then obtained by defining $\mu = \sum_e 2^{-(e+1)} \mu_e$.

If Turing machine $e$ does not halt, $\mu_e([\sigma]) = 1/2$ if $\sigma$ is a string of $0$'s or a string of $1$'s, and $\mu_e([\sigma]) = 0$ otherwise. Suppose, on the other hand, that Turing machine $e$ halts in $s$ steps, and suppose $k < s$. Then there are $2(k-1)$ additional strings $\sigma$ with length $k$ such that $\mu_e([\sigma]) > 0$, each consisting of a string of 1's followed by a string of 0's or vice versa. For each of these $\sigma$, $\mu_e([\sigma]) = 1 / 4s$, and if $\sigma$ is a string of $0$'s or a string of $1$'s of length $k$, $\mu_e([\sigma]) = 1/2 - (k-1)/4s$. Thus when $s$ is large compared to $k$, the non-halting case provides a good approximation to $\mu_e([\sigma])$ when $\len(\sigma) \leq k$, even though $e$ eventually halts. Thus, to compute $\mu_e([\sigma])$ to within $\varepsilon$, it suffices to simulate the $e$th Turing machine $O(k/\varepsilon)$ steps. If it halts before then, that determines $\mu_e$ exactly; otherwise, the non-halting approximation is close enough.

Now consider $f = \lim_n A_n 1_{[1]}$ in $(2^\omega, \mdl B, \mu_e)$. Note that $(A_n 1_{[1]})(\omega)$ counts the density of $1$'s among the first $n$ bits of $\omega$. If Turing machine $e$ does not halt, $f(\omega) = 1$ if $\omega$ is the sequence of 1's, and $f(\omega) = 0$ if $\omega$ is the sequence of 0's. Up to a.e.~equivalence, these are all that matters, since the mass concentrates on these two elements of Cantor space. If Turing machine $e$ halts in $s$ steps, then $f(\omega) = 1/4$ on the shifts of $(1^s0^{3s})^*$, and $f(\omega) = 3/4$ on the shifts of $(1^{3s}0^s)^*$.

Suppose $g$ is $\mdl B_k$-measurable. If Turing machine $e$ halts in $s$ steps and $k$ is much less than $s$, then roughly $3/4$ of the shifts of $(1^s0^{3s})^*$ lie in $[0^k]$ and roughly $1/4$ lie in $[1^k]$; and roughly $3/4$ of the shifts of $(1^{3s}0^s)^*$ lie in $[1^k]$ and roughly $1/4$ lie in $0^k$. But $f(\omega)$ only takes on the values $1/4$ and $3/4$, and $g$ is constant on $[0^k]$ and $[1^k]$. So if $k$ is much less than $s$, $\mu_e (\{ \omega \st | f(\omega) - g(\omega) | > 1/8 \}) > 1 / 4$. Turning this around, given the information that $\mu_e (\{ \omega \st | f(\omega) - g(\omega) | > 1/8 \}) \leq 1 / 4$ for some $g$ of complexity at most $k$ enables one to determine whether or not Turing machine $e$ halts; namely, one simulates the Turing machine for $O(k)$ steps, and if it hasn't halted by then, it never will.

Set $\mu = \sum_e 2^{-(e+1)} \mu_e$. Since, for any $g$,  
\[
\mu_e (\{ \omega \st | f(\omega) - g(\omega) | > 1/8 \}) \leq \mu (\{ \omega \st | f(\omega) - g(\omega) | > 1/(8 \cdot 2^{e+1}) \}),
\]
knowing a $k_e$ for each $e$ with the property that $\mu (\{ \omega \st | f(\omega) - g(\omega) | > 1/(8 \cdot 2^{e+1}) \} < 1 / 4$ for some $g$ of complexity at most $k_e$ enables one to solve the halting problem. But such a $k_e$ can be obtained from a bound on the complexity of $f$. Thus $\mu$ satisfies the statement of the theorem.
\end{proof}

The proof above relativizes, so for any set $X$ there is a measure $\mu$ on $2^\NN$, computable from $X$, such that no bound on the rate of complexity of $f$ can be computed from $X$. As the following corollary shows, this implies that $\lim_n A_n 1_{[1]}$ can have arbitrarily high complexity.

\newtheorem*{corollary}{Corollary}

\begin{corollary}
  For any $v : \QQ^+ \to \NN$ there is a measure $\mu$ on $2^\NN$ such that if $f = \lim_n A_n 1_{[1]}$ and $k(\varepsilon)$ is a bound on the complexity of $f$, then $\limsup_{\varepsilon \to 0} k(\varepsilon) / v(\varepsilon) = \infty$. 
\end{corollary}

\begin{proof}
  Let $\mu$ be such that no bound on the complexity of $f$ can be computed from $v$. If the conclusion failed for some $k$, then there would be a rational $\varepsilon' > 0$ and $N$ such that for every $\varepsilon < \varepsilon'$, $k(\varepsilon) < N \cdot v(\varepsilon)$. But then $k'(\varepsilon) = N \cdot v(\min(\varepsilon, \varepsilon'))$ would be a bound on the complexity of $f$ that is computable from $v$, contrary to our choice of $\mu$.
\end{proof}


\end{document}